\documentclass[leqno,11pt]{amsart}

\usepackage[bookmarksnumbered]{hyperref}	
\usepackage[ngerman,USenglish]{babel}		
\usepackage{color}				
\usepackage[ascii]{inputenc}			
\usepackage{aliascnt}				
\usepackage{amssymb}				
\usepackage{microtype}				
\usepackage[foot]{amsaddr}
\usepackage[all]{xy}
\usepackage{array}
\usepackage[table]{xcolor}
\usepackage{changes}
\usepackage{subdepth}
\usepackage{enumitem} 

\newcommand{\newjointcountertheorem}[3]{
	\newaliascnt{#1}{#2}
	\newtheorem{#1}[#1]{#3}
	\aliascntresetthe{#1}	
}

\newtheorem{thm}{Theorem}[section]
\newjointcountertheorem{satz}{thm}{Satz}
\newjointcountertheorem{lem}{thm}{Lemma}
\newjointcountertheorem{cor}{thm}{Corollary}
\newjointcountertheorem{prp}{thm}{Proposition}
\newjointcountertheorem{cnj}{thm}{Conjecture}
\newjointcountertheorem{que}{thm}{Question}
\newjointcountertheorem{fct}{thm}{Fact}
\newjointcountertheorem{obs}{thm}{Observation}
\theoremstyle{definition}
\newjointcountertheorem{dfn}{thm}{Definition}
\newjointcountertheorem{ntn}{thm}{Notation}
\newjointcountertheorem{rem}{thm}{Remark}
\newjointcountertheorem{nte}{thm}{Note}
\newjointcountertheorem{exl}{thm}{Example}

\DeclareMathOperator{\Res}{Res}
\DeclareMathOperator{\Syl}{Syl}
\DeclareMathOperator{\wt}{wt}

\newcommand{\BF}{\mathbb{F}_2}
\newcommand{\BFn}{\mathbb{F}_{2^n}}

\def\Snospace~{\S{}}

\numberwithin{equation}{section}
\allowdisplaybreaks[4]				
\hypersetup{pdfauthor={Stavros Kousidis \and Andreas Wiemers},pdftitle={On the first fall degree of summation polynomials}}

\colorlet{Changes@Color}{red}

\begin{document}

    
  \title{On the first fall degree of summation polynomials}
  \author{Stavros Kousidis and Andreas Wiemers}
  \address{Federal Office for Information Security, Godesberger Allee 185--189, 53175 Bonn, Germany}
  \email{st.kousidis@googlemail.com}
  
  \date{\today}


  \begin{abstract}
    We improve on the first fall degree bound of polynomial systems that arise
    from a Weil descent along Semaev's summation polynomials relevant to the solution of the Elliptic Curve
    Discrete Logarithm Problem via Gr{\"o}bner basis algorithms.
    
    \bigskip
    \noindent {K\tiny EYWORDS.} Polynomial systems, Gr{\"o}bner bases, Discrete logarithm problem, Elliptic curve cryptosystem.
    
    \bigskip
    \noindent {2010 M\tiny ATHEMATICS \footnotesize S\tiny UBJECT \footnotesize C\tiny LASSIFICATION.} 13P15 Solving polynomial systems, 13P10 Gr{\"o}bner bases, 14H52 Elliptic curves.
  \end{abstract}

  \maketitle

  \section{Introduction}
       
    Finding solutions to algebraic equations is a fundamental task.
    A common approach is a Gr{\"o}bner basis computation via an
    algorithm such as Faug\`ere's $F4$ and $F5$ (see \cite{Faugere99,Faugere02}).
    In recent applications Gr{\"o}bner basis techniques have become relevant
    to the solution of the Elliptic Curve Discrete Logarithm Problem (ECDLP).
    Here one seeks solutions to polynomial equations arising from a
    Weil descent along Semaev's summation polynomials \cite{Semaev04} which
    represents a crucial step in an index calculus
    method for the ECDLP, see e.g.~\cite{PetitQ12,Semaev15}.
    The efficiency of Gr{\"o}bner basis algorithms is governed by a so-called
    \textit{degree of regularity}, that is the highest degree
    occurring along the subsequent computation of algebraic relations.
    It is widely believed that this often intractable
    complexity parameter is closely approximated by the degree of the first
    non-trivial algebraic relation, the \textit{first fall degree}.
    In particular, the algorithms for the ECDLP of Petit and Quisquater
    \cite{PetitQ12} are sub-exponential
    under the assumption that this approximation is in o$(1)$.
    
    In the present paper, we will improve Petit's and Quisquater's \cite{PetitQ12}
    first fall degree bound $m^2 +1$ for the system arising from
    the Weil descent along Semaev's $(m+1)$-th summation polynomial.
    That is, we prove that a degree fall occurs at degree
    $m^2-m+1$ by exhibiting the highest degree homogeneous part of that polynomial system.
    In fact, this degree is $m^2 - m$, so that we expect the bound to be sharp except
    for the somewhat pathological case $m=2$ that has been discussed by
    Kosters and Yeo \cite{KostersY15}. This allows us to sharpen the asymptotic run time of the index calculus algorithm for the ECDLP as exhibited in the complexity analysis of Petit and Quisquater \cite{PetitQ12}.
  
  \section{The first fall degree}
  \label{sec:first-fall-degree}

    The notion of the first fall has been described by Faug\`ere and Joux \cite[Section 5.1]{faugerejoux2003}, Granboulan, Joux and Stern \cite[Section 3]{GranboulanJS06}, Dubois and Gama \cite[Section 2.2]{duboisgama2010}, and Ding and Hodges \cite[Section 3]{dinghodges2011}. Although the concept of the first fall degree has been called \textit{minimal degree} \cite{faugerejoux2003} and \textit{degree of regularity} \cite{dinghodges2011, duboisgama2010, GranboulanJS06}, we actually adopt the terminology and definition of Hodges, Petit and Schlather \cite{HodgesPS14}. For readability reasons we include a brief and tailored account of the first fall degree and refer the reader to \cite[Section 2]{HodgesPS14} for details and greater generality.
    
    Our considerations take place over a degree $n$ extension $\BFn$ of the binary field $\BF$. Consider the decomposition of the graded ring
    \[
    		S = \BFn[X_0,\ldots,X_{N-1}]/(X_0^2,\ldots,X_{N-1}^2)
    \]
    into its homogeneous components
    \[
      S = S_0 \oplus S_1 \oplus \cdots \oplus S_N .
    \]
    Each $S_j$ is the $\BFn$-vector space generated by the monomials of degree $j$.
    Let $I$ be an ideal in $S$ generated by homogeneous polynomials $h_1,\ldots,h_r \in S_d$ all of the same degree $d$. Then we have a surjective map
    \[
    \begin{array}{cccl}
      \phi : & S^r & \longrightarrow & I \\
	     & (g_1,\ldots,g_r) & \mapsto & g_1 h_1 + \cdots + g_r h_r  .
    \end{array}
    \]
    Without loss of generality we furthermore assume $0 < r = \dim_{\BFn} \sum_{j=1}^r \BFn h_j$. Let $e_i$ denote the canonical $i$-th basis element of the free $S$-module $S^r$.
    The $S$-module $U$ generated by the elements
    \begin{align*}
      h_j e_i + h_i e_j \mbox{ and } h_k e_k \mbox{, where $i,j,k = 1,\ldots,r$,}
    \end{align*}
     is a subset of $\ker (\phi)$. If we restrict $\phi$ to the $\BFn$-subvector space $S^r_{j-d} \subset S^r$
    we obtain a surjective map
    \[
    \begin{array}{cccl}
      \phi_{j-d} : & S^r_{j-d} & \longrightarrow & I \cap S_j
    \end{array}
    \]
    whose kernel contains the $\BFn$-subvector space $U_{j-d} = U \cap S^r_{j-d}$ and hence factors through
    \[
      \bar{\phi}_{j-d} :  S^r_{j-d} / U_{j-d} \rightarrow  I \cap S_j .
    \]
    \begin{dfn}[\protect{Cf.~\cite[Definition 2.1]{HodgesPS14}}]
    \label{dfn:first-fall-degree}
      The first fall degree of a homogeneous system $h_1,\ldots,h_r \in S_d$ and its linear span $\sum_{j=1}^r \BFn h_j$, respectively, is the smallest $j$ such that the induced
      $\BFn$-linear map $\bar{\phi}_{j-d}$ is not injective,
      that is the smallest $j$ such that $\dim_{\BFn} (I \cap S_j) < \dim_{\BFn} (S^r_{j-d} / U_{j-d})$.
      It is denoted by $D_{ff}(\sum_{j=1}^r \BFn h_j)$.
      
    \end{dfn}
	Following \cite{HodgesPS14} we now consider the ring of functions
	\[
		A_{\BFn} = \BFn[X_0,\ldots,X_{N-1}]/(X_0^2-X_0,\ldots,X_{N-1}^2-X_{N-1})
	\]
	as a finite-dimensional filtered algebra whose filtration components $[A_{\BFn}]_d$, $d \in \mathbb{N}$, are given by the polynomials up to degree $d$. The associated graded ring of $A_{\BFn}$ is
	\[
    		\mathrm{Gr}(A_{\BFn}) = {\BFn}[X_0,\ldots,X_{N-1}]/(X_0^2,\ldots,X_{N-1}^2)
    \]
	whose graded components
	\[
		[\mathrm{Gr}(A_{\BFn})]_d = [A_{\BFn}]_d / [A_{\BFn}]_{d-1} \mbox{, for $d \in \mathbb{N}$},
	\]
	are given by the homogeneous polynomials of degree $d$. Any linear subspace $V \subset [A_{\BFn}]_d$ induces a homogeneous linear subspace $\bar{V} \subset [\mathrm{Gr}(A_{\BFn})]_d$ via the canonical projection $\pi_d : [A_{\BFn}]_d \rightarrow [\mathrm{Gr}(A_{\BFn})]_d$.
	\begin{dfn}[\protect{Cf.~\cite[Definition 2.2]{HodgesPS14}}]
    \label{dfn:first-fall-degree-}
       Consider a polynomial system $p_1,\ldots,p_r \in [A_{\BFn}]_d$ and its linear span $V = \sum_{j=1}^r \BFn p_j \subset [A_{\BFn}]_d$, respectively. We assume without loss of generality that $\dim_{\BFn}V = r>0$. The first fall degree of $V$ is
      \[
    			D_{ff} \left( V \right) =
	    		\begin{cases}
    				d & \mbox{, $\dim_{\BFn}\bar{V} < \dim_{\BFn}V$,}\\
    				D_{ff}(\bar{V}) & \mbox{, else,}
	    		\end{cases}
	  \]
	  where $D_{ff}(\bar{V} = \sum_{j=1}^r \BFn \pi_d(p_j))$ is given in \autoref{dfn:first-fall-degree}.
    \end{dfn}
	
  \section{Weil descent along summation polynomials}
  
    We prove that the first fall degree of the polynomial system
    that arises from a Weil descent along Semaev's summation polynomial $S_{m+1}$ is bounded from above by $m^2-m +1$.
    This is an improvement over $m^2 +1$ that results from \cite[Theorem 5.2]{HodgesPS14} and \cite[Section 4]{PetitQ12}. Let us briefly introduce the summation polynomials and describe the Weil descent.
    
     Semaev \cite{Semaev04} introduced the $m$-th summation polynomial $S_m(x_1,\ldots,x_m) \in \mathbb{K}[x_1,\ldots,x_m]$ on an elliptic curve $E: y^2 = x^3 +a_4 x + a_6$ over a finite field $\mathbb{K}$ with $\mathrm{char}(\mathbb{K}) \neq 2,3$ by the following defining property: for elements $x_1,\ldots,x_m$ in the algebraic closure $\bar{\mathbb{K}}$ one has $S_m(x_1,\ldots,x_m) = 0$ if and only if there exist $y_1,\ldots,y_m \in \bar{\mathbb{K}}$ such that $(x_1,y_1),\ldots,(x_m,y_m) \in E(\bar{K})$ and $(x_1,y_1) + \ldots + (x_m,y_m) = 0$ on $E$. Semaev gave a recursive formula based on resultants to compute those polynomials and described some properties \cite[Theorem 1]{Semaev04}. The summation polynomials can also be given in characteristic $2$. We consider $\mathbb{K} = \BFn$, an ordinary, i.e.~non-singular, elliptic curve $E : y^2 + xy = x^3 +a_2 x^2 +a_6$, and the projection to the $x$-coordinate $x(P_i) = x(x_i,y_i) = x_i$ of $P_i \in E$ . Then, still
      \begin{align*}
      	S_2(x_1,x_2) & = x_1 - x_2
	\end{align*}
	and from Diem's general description \cite[Lemma 3.4, Lemma 3.5]{diem2011} one can deduce
	 \begin{align*}
		S_3(x_1,x_2,x_3) & = (x_1^2+x_2^2) x_3^2 + x_1 x_2 x_3 + x_1^2x_2^2 + a_6 \\
		S_{m+1}(x_1\ldots,x_m,x_{m+1}) & = \Res_X(S_m(x_1,\dots,x_{m-1},X),S_3(x_m,x_{m+1},X)) 
      \end{align*}
      and the degree of $S_{m+1}$ in each variable $x_i$ is $2^{m-1}$. Note that these formulas have also been outlined by Petit and Quisquater \cite[Section 5]{PetitQ12} 
      who also refer to Diem \cite{diem2011}.
      
     To describe the Weil descent along those summation polynomials (see e.g.~\cite[Section 4]{PetitQ12})
     we fix a basis $1, z, \ldots, z^{n-1}$ of $\BFn$ over $\BF$ and
     let $W$ be a subvector space in $\BFn$ of dimension $n'$ and
     basis $\nu_1,\ldots,\nu_{n'}$ over $\BF$.
     We introduce $m n'$ variables $y_{ij}$ that model the linear constraints
    \[
    		x_i = \sum_{l=1}^{n'} y_{il} \nu_l ,
    	\]
    	set $x_{m+1}$ to an arbitrary element $c \in \BFn$, and obtain the equation system
    \begin{align*}
      S_{m+1}(x_1,\ldots,x_m,c) & = S_{m+1} \left( \sum_{l=1}^{n'} y_{1l} \nu_l, \ldots , \sum_{l=1}^{n'} y_{ml} \nu_l, c \right) \\
	& = f_0(y_{ij}) + z f_1(y_{ij}) + \cdots + z^{n-1} f_{n-1}(y_{ij})
    \end{align*}
    The first fall degree of interest is that of the reduced polynomial system
    \begin{align}
    \label{lbl:polys-from-weil-descent}
      s_k \equiv f_k \bmod (y_{11}^2 - y_{11}^{ }, \ldots , y_{mn'}^2 - y_{mn'}^{ }) \mbox{, where } k = 0,\ldots , n-1 .
    \end{align}
    Note that $s_0,\ldots,s_{n-1} \in \BF[y_{11},\ldots,y_{m n'}]/(y_{11}^2-y_{11},\ldots, y_{m n'}^2-y_{m n'})$.
    
    By the definition of the first fall degree we are interested in the highest degree homogeneous part of $s_0,\ldots, s_{n-1}$
    whose degree can be determined as follows.
    
    \begin{lem}
    \label{lem:summation-monomial}
      Let $m \geq 3$. The highest degree homogeneous part of the polynomial system $s_0,\ldots,s_{n-1} \in \BF[y_{11},\ldots,y_{m n'}]/(y_{11}^2-y_{11},\ldots, y_{m n'}^2-y_{m n'})$ from \autoref{lbl:polys-from-weil-descent} is induced by the monomial
      \[
      	(x_1 \cdots x_m)^{2^{m-1}-1} \cdot x_{m+1}
 	  \]
      in the summation polynomial $S_{m+1}(x_1,\ldots,x_m,x_{m+1})$, and hence its degree is less than or equal to $m^2-m$.
    \end{lem}
    
    \begin{proof}
      First, we show the existence of the monomial $(x_1 \cdots x_m)^{2^{m-1}-1} \cdot x_{m+1}$ in $S_{m+1}(x_1,\ldots,x_m,x_{m+1})$. We have
      \begin{align*}
	S_3(x_1,x_2,x_3) & = (x_1^2+x_2^2) x_3^2 + x_1 x_2 x_3 + x_1^2x_2^2 + a_6 \\
	S_{m+1}(x_1\ldots,x_m,x_{m+1}) & = \Res_X(S_m(x_1,\dots,x_{m-1},X),S_3(x_m,x_{m+1},X)) 
      \end{align*}
      and the degree of $S_{m+1}$ in each variable $x_i$ is $2^{m-1}$.
      The resultant of $f,g \in \BFn[X]$ of degree $k$ and $l$ is the determinant of the Sylvester matrix
      \begin{align*}
	\Res_X(f,g) & = \det \big( \Syl (f,g) \big)  \\
	  & = \det
	    \begin{pmatrix}
	      f_k &     & \cdots &        & f_0    &        &        &     \\
		  & f_k &        & \cdots &        &  f_0   &        &     \\
		  &     & \ddots &        &        &        & \ddots &     \\
		  &     &        &  f_k   &        & \cdots &        & f_0 \\
	      g_l &     & \cdots &        &  g_0   &        &        &     \\
		  & g_l &        & \cdots &        &  g_0   &        &     \\
		  &     & \ddots &        &        &        & \ddots &     \\
		  &     &        & g_l    &        & \cdots &        & g_0  \\
	    \end{pmatrix}
      \end{align*}
      That is, with
      \begin{align*}
	S_3(x_m,x_{m+1},X) & = (x_m^2+x_{m+1}^2) X^2 + x_mx_{m+1} X + x_m^2x_{m+1}^2 + a_6\\
	S_m(x_1,\dots,x_{m-1},X) & = c_{2^{m-2},m}X^{2^{m-2}} + \cdots + c_{0,m}
      \end{align*}
      where each $c_{i,m} \in \BFn[x_1,\ldots,x_{m-1}]$, we have
      \begin{align*}
	S_{m+1}(x_1\ldots,x_m,x_{m+1}) = \det \big( \Syl (S_m,S_3) \big) .
	\end{align*}
      To be concrete, $\Syl (S_m,S_3)$ is the matrix
       \begin{align*}
       \begin{pmatrix}
	  c_{2^{m-2},m} &  c_{2^{m-2}-1,m}   & \cdots &    c_{0,m}    &  0      &             \\
	  0 & c_{2^{m-2},m} &    \cdots    &   c_{1,m}     &  c_{0,m}   &             \\
	  x_m^2+x_{m+1}^2 &  x_mx_{m+1}   & x_m^2x_{m+1}^2 + t &  &        &                       \\
	  & \ddots &        &   \ddots       &       &      \\
	  &        & x_m^2+x_{m+1}^2    &   x_mx_{m+1}     & x_m^2x_{m+1}^2 + t          \\
	\end{pmatrix}
	\end{align*}
      \noindent with a total of $2^{m-2}+2$ rows and columns.
      In order to prove our claim we have to
      identify specific summands in the Leibniz formula of the determinant. That is, we consider
      \begin{align}
      \label{eqn:leibniz}
      	\det \big( \Syl (S_m,S_3) \big) = \sum_{\pi} \mathrm{sgn}(\pi) \prod_{i=1}^{2^{m-2}+2} \Syl (S_m,S_3)_{i,\pi_i}
      \end{align}
      and argue that for the relevant summands no cancellation over $\BFn$ occurs. Note that the sign of a permutation is $1 \in \BFn$.
      
      \noindent Step $1$: Prove by induction (start with $x_1^2x_2^2$ in $S_3$) that $S_{m+1}$ contains the monomial $(x_1 \cdots x_m)^{2^{m-1}}$ in its term $c_{0,m+1}$.
	  For that we consider the permutation
	  \begin{align}
	  \label{eqn:permutation-sigma}
	  	\sigma & = \Big( \sigma_1,\ldots , \sigma_{2^{m-2}+2} \Big) \\
	  		& = \left(2^{m-2}+1,2^{m-2}+2,1,2,\ldots,2^{m-2} \right) \notag
	  \end{align}
	  and obtain
	  \begin{align*}
		S_{m+1}(x_1\ldots,x_m,x_{m+1}) &  = \mathrm{sgn}(\sigma) \prod_{i=1}^{2^{m-2}+2} \Syl (S_m,S_3)_{i,\sigma_i} + \cdots \\
		& = c_{0,m} c_{0,m} \prod_{i=1}^{2^{m-2}} (x_m^2+x_{m+1}^2) + \ldots \\
	 	& = \left( (x_1 \cdots x_{m-1})^{2^{m-2}} \right)^2 \cdot x_m^{2^{m-1}} + \ldots \\
	  	& = (x_1 \cdots x_{m-1} x_m)^{2^{m-1}} + \ldots
      \end{align*}
      Note that specifying $\sigma_1 = 2^{m-2}+1$ and $\sigma_2 = 2^{m-2}+2$ determines $\sigma$ since the remaining entries in $\Syl (S_m,S_3)$ form an upper triangular matrix with $x_m^2+x_{m+1}^2$ on the diagonal.
      
      \noindent Step $2$: Prove by induction (start with $x_1x_2x_3$ in $S_3$) that $S_{m+1}$ contains the monomial $(x_1 \cdots x_m)^{2^{m-1}-1} \cdot x_{m+1}$, i.e.~$(x_1 \cdots x_m)^{2^{m-1}-1}$ in its term $c_{1,m+1}$.
	  For that we consider the permutation
	  \begin{align}
  	  \label{eqn:permutation-tau}
	  	\tau & = \Big( \tau_1, \ldots , \tau_{2^{m-2} + 2} \Big) \\
	  		& = \left( 2^{m-2},2^{m-2}+2,1,\ldots,2^{m-2}-1,2^{m-2}+1 \right) \notag
	  \end{align}
	  and  obtain
      \begin{align*}
	& S_{m+1}(x_1\ldots,x_m,x_{m+1}) \\
	& = \mathrm{sgn}(\tau) \prod_{i=1}^{2^{m-2}+2} \Syl (S_m,S_3)_{i,\tau_i} + \cdots \\
	& = c_{1,m} c_{0,m} \cdot x_m x_{m+1} \prod_{i=1}^{2^{m-2}-1} (x_m^2+x_{m+1}^2) + \ldots \\
	& = (x_1 \cdots x_{m-1})^{2^{m-2}-1} \cdot (x_1 \cdots x_{m-1})^{2^{m-2}} \cdot x_m x_{m+1} (x_m^2)^{2^{m-2}-1} + \ldots \\
	& = (x_1 \cdots x_{m-1} x_m)^{2^{m-1}-1} \cdot x_{m+1} + \ldots
      \end{align*}
      Note that specifying $\tau_1 = 2^{m-2}$ and $\tau_2 = 2^{m-2}+2$ determines $\tau$ since the remaining entries in $\Syl (S_m,S_3)$ form an upper triangular matrix with $x_m^2+x_{m+1}^2, \ldots , x_m^2+x_{m+1}^2, x_m x_{m+1}$ on the diagonal.
	  
      Second, in order to exclude potential cancellations we have to show that the permutations $\sigma$ in \eqref{eqn:permutation-sigma} and $\tau$ in \eqref{eqn:permutation-tau} are the only possible choices to produce the monomials $(x_1 \cdots x_m)^{2^{m-1}}$ and $(x_1 \cdots x_m)^{2^{m-1}-1} \cdot x_{m+1}$ in $S_{m+1}$, respectively. For that, we prove by induction (start with $x_1x_2$ in $S_3$) that the only multiples of $(x_1 \cdots x_m)^{2^{m-1}-1}$ in the coefficients of $S_{m+1}$ are $(x_1 \cdots x_m)^{2^{m-1}}$ in $c_{0,m+1}$ and $(x_1 \cdots x_m)^{2^{m-1}-1}$ in $c_{1,m+1}$. Indeed, the factor $(x_1 \cdots x_{m-1})^{2^{m-1}-1}$ in the variables $x_1,\ldots,x_{m-1}$ can only be produced by products $c_{i,m} \cdot c_{j,m}$ of entries taken from the first two rows of the Sylvester matrix $\Syl (S_m,S_3)$. Since the degree of $S_{m}$ in each variable $x_1,\ldots,x_{m-1}$ is $2^{m-2}$, each of the entries $c_{0,m}, \ldots, c_{2^{m-2},m}$ is a sum of monomials in the variables $x_1,\ldots,x_{m-1}$ where each monomial is either
      \begin{enumerate}[label=(\roman*)]
      	\item no multiple of $(x_1 \cdots x_{m-1})^{2^{m-2}-1}$ or
      	\item a multiple $( x_1 \cdots x_{m-1})^{2^{m-2}-1} \cdot x_1^{\delta_1}\cdots x_{m-1}^{\delta_{m-1}}$, with $\delta_i \in \{0,1\}$. 
      \end{enumerate}
	Therefore, the monomials in the products $c_{i,m}\cdot c_{j,m}$ that contribute to the determinant \eqref{eqn:leibniz} occur in the following forms
\begin{eqnarray}
	\label{eqn:1} && ( ( x_1 \cdots x_{m-1})^{2^{m-2}-1})^2 \cdot x_1^{\delta_1 + \delta'_1}\cdots x_{m-1}^{\delta_{m-1}+\delta'_{m-1}} \\
	\label{eqn:2} && ( x_1 \cdots x_{m-1})^{2^{m-2}-1} \cdot x_1^{\delta_1}\cdots x_{m-1}^{\delta_{m-1}} \cdot \mu \\
	\label{eqn:3} && \mu \cdot \mu' 
\end{eqnarray}
where $\mu$ and $\mu'$ denote elements that are no multiples of $(x_1 \cdots x_{m-1})^{2^{m-2}-1}$.
Consequently, a monomial in the product $c_{i,m}\cdot c_{j,m}$ that is now a multiple of $(x_1 \cdots x_{m-1})^{2^{m-1}-1}$ can only arise in case \eqref{eqn:1} if for each $k=1,\ldots,m-1$ the following condition holds
\[
	2 \cdot \left( {2^{m-2}-1} \right) + \delta_k + \delta'_k \geq 2^{m-1}-1 \iff \delta_k + \delta'_k \geq 1 .
\]
Due to the degree restriction of $S_{m}$ a product $c_{i,m}\cdot c_{j,m}$ where the monomials in $c_{i,m}$ and $c_{j,m}$ are all of the form \eqref{eqn:2} or \eqref{eqn:3} cannot produce a multiple of $(x_1 \cdots x_{m-1})^{2^{m-1}-1}$. Therefore, we are left with products of the terms $c_{0,m}$ and $c_{1,m}$ by the induction hypothesis.
Since $c_{1,m}\cdot c_{1,m}$ only produces $(x_1 \cdots x_{m-1})^{2^{m-1}-2}$, the permutations $\pi = (\pi_1,\pi_2,\ldots,\pi_{2^{m-2}+2})$ in the Leibniz formula \eqref{eqn:leibniz} that produce multiples of the monomial $(x_1 \cdots x_{m-1})^{2^{m-1}-1}$ must have either $(\pi_1, \pi_2) = (\sigma_1, \sigma_2)$ or $(\pi_1,\pi_2) = (\tau_1,\tau_2)$ as given in \eqref{eqn:permutation-sigma} and \eqref{eqn:permutation-tau}, respectively. This determines our permutations $\sigma$ and $\tau$ completely.

      To finish the proof, our degree claim in \autoref{lem:summation-monomial} is argued as follows. The variables $y_{ij}$ of the $s_k$ are over $\BF$ where taking squares is a linear operation. Therefore the degrees of the homogeneous parts of the system $s_0,\ldots,s_{n-1}$ depend only on the Hamming weight 
      $\wt (x_1^{\alpha_1} \cdots x_m^{\alpha_m}) = \sum \wt(\alpha_i)$ of a monomial in $S_{m+1}$.
      Since the degree of $S_{m+1}$ in each variable $x_i$ is $2^{m-1}$ the monomial $(x_1 \cdots x_m)^{2^{m-1}-1} \cdot x_{m+1}$,
      when $x_{m+1}$ is set to an element $c \in \BFn$, produces the highest Hamming weight $\sum_{i=1}^m \wt(2^{m-1}-1) = m(m-1)$.
	  To be precise, we consider
      \[
		x_i^{2^j} = (\sum_{l=1}^{n'} y_{il} \nu_l )^{2^j} = \sum_{l=1}^{n'} y_{il} \nu_l^{2^j}
      \]
      and obtain
      \begin{equation}
      \label{eq:special-element}
	(x_1 \cdots x_m)^{2^{m-1}-1} \cdot c = c \prod_{i=1}^m \prod_{j=0}^{m-2} \sum_{l=1}^{n'} y_{il} \nu_l^{2^j}
      \end{equation}
      which is of degree less than or equal to $m(m-1)$ in the variables $y_{ij}$.
    \end{proof}
    
    We are ready to prove the main result.
    
    \begin{thm}
    \label{thm:semaev-first-fall}
    	  Let $n' \geq m \geq 3$ and $c \in \BFn \setminus \{0\}$, and consider the
    	  polynomial system $s_0,\ldots,s_{n-1} \in \BF[y_{11},\ldots,y_{m n'}]/(y_{11}^2-y_{11},\ldots, y_{m n'}^2-y_{m n'})$ from \autoref{lbl:polys-from-weil-descent}, that results from the Weil descent along the summation polynomial $S_{m+1}(x_1,\ldots,x_m,c)$.
      The first fall degree of $s_0,\ldots,s_{n-1}$ is less than or equal to $m^2-m+1$.
    \end{thm}
    
    \begin{proof}
      	Consider the finite-dimensional filtered algebra
      	\[
			A_{\BF} = \BF[y_{11},\ldots,y_{m n'}]/(y_{11}^2-y_{11},\ldots, y_{m n'}^2-y_{m n'}) .
		\]      	      	
      	The linear span 
      	\[
			\sum_{j=0}^{n-1} \BF s_j
        \]
    		is inside the degree $d=m^2-m$ subspace of the filtered algebra $A_{\BF}$ due to
    		\autoref{lem:summation-monomial}. By \cite[Corollary 2.4]{HodgesPS14}
    		an extension of the base field, i.e.
  		\[
        		A_{\BFn} = \BFn[y_{11},\ldots,y_{m n'}]/(y_{11}^2-y_{11},\ldots, y_{m n'}^2-y_{m n'}) ,
        	\]
        	does not affect the first fall degree. That is,
		\[    		
    		 	D_{ff}\left(\sum_{j=0}^{n-1} \BF s_j \right) = D_{ff}\left( \sum_{j=0}^{n-1} \BFn s_j \right) .
    		\]
    		By \cite[Definition 2.2]{HodgesPS14}, the first fall degree of the subspace $\sum_{j=0}^{n-1} \BFn s_j$ of $A_{\BFn}$ is
    		\[
    			D_{ff} \left( \sum_{j=0}^{n-1} \BFn s_j \right) =
	    		\begin{cases}
    				d=m^2-m & \mbox{, $\dim_{\BFn}\bar{V} < \dim_{\BFn}V$,}\\
    				D_{ff}(\bar{V}) & \mbox{, else,}
	    		\end{cases}
	    	\]
    		where $\bar{V}$ denotes the induced homogeneous subspace of $\sum_{j=0}^{n-1} \BFn s_j$ in the associated graded ring
    		\[
			\mathrm{Gr}(A_{\BFn}) = \BFn[y_{11},\ldots,y_{m n'}]/(y_{11}^2,\ldots, y_{m n'}^2) .
		\]
		If $\dim_{\BFn}\bar{V} < \dim_{\BFn}V$, our claim follows.
		Otherwise we consider the polynomial
	 	\[
			P_0 = c \prod_{i=1}^m \prod_{j=0}^{m-2} \sum_{l=1}^{n'} y_{il} \nu_l^{2^j}
	    \]
	    which is an element of the homogeneous subspace $\bar{V}$ by \autoref{lem:summation-monomial}, and in particular \autoref{eq:special-element}. Now, for any
    		\[
    			x_k = \sum_{l=1}^{n'} y_{kl} \nu_l
	    	\]
  	 	we have a non-trivial relation
   	 	\[
		    	x_k P_0=c \sum_{l=1}^{n'} y_{kl}^2 \nu_l^2 \cdot \prod_{j=1}^{m-2} \sum_{l=1}^{n'} y_{kl} \nu_l^{2^j} \cdot \prod_{i=1, i\neq k}^m \prod_{j=0}^{m-2} \sum_{l=1}^{n'} y_{il} \nu_l^{2^j} = 0 \in \mathrm{Gr}(A_{\BFn})
		\]
		of degree $d+1=m^2-m+1$ unless $P_0 = 0 \in \mathrm{Gr}(A_{\BFn})$. Therefore it remains to show that $P_0 \neq 0$. For that purpose, we recall that $c \in \BFn \setminus \{0\}$, $v_1, \ldots, v_{n'}$ are linearly independent, and $n' \geq m$. Consider the linear change of variables
      \[
	Y_{ij} = x_i^{2^j} = (\sum_{l=1}^{n'} y_{il} \nu_l )^{2^j} = \sum_{l=1}^{n'} y_{il} \nu_l^{2^j} .
      \]
      This is induced by the $m \times n'$ matrix
      \[
	      \renewcommand{\arraystretch}{1.5}
    		  \begin{pmatrix}
			\nu_1 & \cdots & \nu_{n'} \\ 
			\nu_1^2 & \cdots & \nu_{n'}^2 \\ 
			\vdots & \ddots & \vdots \\ 
			\nu_1^{2^{m-2}} & \cdots & \nu_{n'}^{2^{m-2}}
	      \end{pmatrix} 
    		  \renewcommand{\arraystretch}{1}
      \]
      that can be completed to an invertible linear transform by \cite[Lemma 3.51]{LidlN86}, since we have assumed $v_1, \ldots, v_{n'}$ to be linearly independent and $n' \geq m$. By using such an invertible linear transform 
on any block of variables 
$$y_{i1},\ldots,y_{in'}$$      
      we get new 
      variables 
      $$Y_{10}, \ldots, Y_{m,n'-1}.$$
Under this change of variables $P_0$ is mapped to the non-zero element
\[
c \prod_{i=1}^m \prod_{j=0}^{m-2} Y_{ij} \in \BFn[Y_{10},\ldots,Y_{m,n'-1}]/(Y_{10}^2,\ldots, Y_{m,n'-1}^2 ). \qedhere
\]
    \end{proof}
    
    
    \begin{rem}
    \label{rem:ff-m2}
      Our \autoref{thm:semaev-first-fall} remains true also in the case $m=2$ with
      first fall degree less than or equal to $2\cdot1+1 =3$.
      This bound is not sharp though, in fact the first fall degree in the case $m=2$
      equals $2$ \cite[Corollary 4.11 and Remark 4.12]{KostersY15}.
    \end{rem}


	 \begin{table}[ht]
    \caption{Empirical data for the Weil descent along the summation polynomial $S_{m+1}$ over $\BFn$ with $n'$-dimensional factor basis. Displayed are the observed first fall degree $D_{ff}$, degree of regularity $D_{reg}$, the time in seconds $s$ and space requirement in gigabyte GB. All values are averaged over $10$ repetitions. For the case $m=2$ see also \autoref{rem:ff-m2}.}
    \label{table:experiments-semaev}
      \begin{tabular}{| >{$}c<{$} | >{$}c<{$} | >{$}c<{$} | >{$}c<{$} | >{$}c<{$} | >{$}c<{$}| >{$}c<{$}| >{$}c<{$}|}
	\hline
	m & n  & n' & m(m-1)+1 & D_{ff} & D_{reg} & s & \mbox{GB} \\
	\hline
	2 & 34 & 17 & 3   & 2      & 4       & 188    & 1.2   \\
	2 & 35 & 18 & 3   & 2      & 4       & 1237   & 16.1  \\
	2 & 36 & 18 & 3   & 2      & 4       & 1342   & 16.4  \\
	2 & 37 & 19 & 3   & 2      & 5       & 2542   & 29.2  \\
	2 & 38 & 19 & 3   & 2      & 5       & 2815   & 25.2  \\
	2 & 39 & 20 & 3   & 2      & 5       & 4785   & 45.6  \\
	2 & 40 & 20 & 3   & 2      & 5       & 4858   & 46.3  \\
	2 & 41 & 21 & 3   & 2      & 5       & 7930   & 65.3  \\
	2 & 42 & 21 & 3   & 2      & 5       & 8901   & 66.7  \\
	2 & 43 & 22 & 3   & 2      & 5       & 16816  & 95.5  \\
	2 & 44 & 22 & 3   & 2      & 5       & 15690  & 96.8  \\
	2 & 45 & 23 & 3   & 2      & 5       & 38352  & 140.0 \\
	2 & 46 & 23 & 3   & 2      & 5       & 31735  & 140.7 \\
	2 & 47 & 24 & 3   & 2      & 5       & 103200 & 207.7 \\
	2 & 48 & 24 & 3   & 2      & 5       & 86636  & 208.2 \\
	\hline
	3 & 13 & 5  & 7  & 7      & 7       & 14     & 0.6   \\
	3 & 14 & 5  & 7  & 7      & 7       & 14     & 0.7   \\
	3 & 15 & 5  & 7  & 7      & 7       & 14     & 0.7   \\
	3 & 16 & 6  & 7  & 7      & 7       & 597    & 13.5  \\
	3 & 17 & 6  & 7  & 7      & 7       & 656    & 13.3  \\
	3 & 18 & 6  & 7  & 7      & 7       & 729    & 34.1  \\
	3 & 19 & 7  & 7  & 7      & 7       & 16571  & 92.2  \\
	3 & 20 & 7  & 7  & 7      & 7       & 17684  & 101.2 \\
	3 & 21 & 7  & 7  & 7      & 7       & 17681  & 90.2  \\
	\hline      
	4 & 13 & 4  & 13 & 13     & 13      & 467    & 25.0  \\
	4 & 14 & 4  & 13 & 13     & 13      & 487    & 25.8  \\
	4 & 15 & 4  & 13 & 13     & 13      & 592    & 26.3  \\
	4 & 16 & 4  & 13 & 13     & 13      & 755    & 27.6  \\
	\hline      
      \end{tabular}
    \end{table}  
    
	\section{Experiments and Conclusion}
	
    In the light of the first fall degree bound given in \autoref{thm:semaev-first-fall}
    we computed a Gr{\"o}bner basis
    for the ideal resulting from the Weil descent along the summation
    polynomial $S_{m+1}(x_1,\ldots,x_m,x_{m+1})$ for $m=2,3,4$ on an
    AMD Opteron CPU with Magma's \verb+GroebnerBasis()+ function.
    Again, we set the verbose level to $1$ and extracted the empirical
    first fall degree $D_{ff}$
    as the step degree of the first step where new lower degree
    (i.e. less than step degree) polynomials are added.
    The empirical degree of regularity $D_{reg}$ is the highest step degree that
    appears during the Gr{\"o}bner basis computation.
    In each experiment we chose a random non-singular elliptic curve over $\BFn$,
    a random subvector space of dimension $n' = \lceil n/m \rceil$ as the factor basis,
    and set $x_{m+1}$ to the $x$-coordinate of a random point on the curve.
    The experimental results that extend the ones present in the literature by Petit and Quisquater \cite{PetitQ12} and Kosters and Yeo \cite{KostersY15} are displayed in \autoref{table:experiments-semaev}.
    
    Like Kosters and Yeo \cite[Section 5]{KostersY15} we observed a raise in the
    regularity degree for $m=2$ in our experiments
    and were able to verify their observation that with the low degree
    polynomials $W = \mbox{span} \{ 1,z,\ldots,z^{n'} \}$
    chosen as the factor basis (Cf. \cite[Section 4.5]{Semaev15}) the raise in the
    regularity degree was produced for slightly
    greater $n=45$.
    It would be very interesting to observe a raise in the degree of regularity
    for higher Semaev polynomials,
    but time and memory amounts become a serious issue for $m\geq3$. 
    However, such observations might neither falsify \cite[Assumption 2]{PetitQ12}
    that $D_{reg} = D_{ff} + \mbox{o}(1)$
    nor lead to further evidence that the gap between the degree of regularity
    and the first fall degree depends on $n$
    as discussed in \cite[Section 5.2]{HuangKY15}.    
    
    However, we believe our first fall degree bound $m^2-m+1$ for Semaev polynomials to be sharp for $m \geq 3$,
    and rephrase \cite[Assumption 2]{PetitQ12} as the following question:
    \begin{align}
    \label{lbl:question}
      D_{reg} = m^2 - m + 1 + \mbox{o}(1) \mbox{ ?}
    \end{align}
    Note that our upper bound on the first fall degree of summation polynomials is a first step
    towards answering this question. The first fall degree generically bounds
    the degree of regularity from below. Hence, any further lower bound on the
    degree of regularity associated to the specific case of a Weil descent along 
    summation polynomials can potentially answer \eqref{lbl:question}.
    
    Assuming an affirmative answer to \eqref{lbl:question}, we can furthermore sharpen
    the asymptotic complexity of the index calculus algorithm
    for the ECDLP as presented by Petit and Quisquater \cite[Section 5]{PetitQ12}.
    In the paragraph \textit{A new complexity analysis} of \cite[Section 5]{PetitQ12}
    it is argued that the complexity of the index calculus approach via summation
    polynomials is dominated by the Gr{\"o}bner basis computation.
    Under the assumption that the degree of regularity is approximated closely
    by the first fall degree \cite[Assumption 2]{PetitQ12}, Petit and Quisquater derive
    \cite[Proposition 4]{PetitQ12}, i.e. that the discrete logarithm can asymptotically
    be solved in sub-exponential time
    \begin{align}
    \label{lbl:runtime}
    		\mathcal{O} \left( 2^{c \log(n) \left( n^{2/3} + 1 \right) } \right),
    	\end{align}
    	where $c = 2 \omega /3$, $\omega$ is the linear algebra constant ($\omega = \log(7)/\log(2)$ is used in the following estimates), and $n^{2/3} + 1$
    	is an upper bound for the first fall degree of the $m$-th summation polynomial when $m=n^{1/3}$ \cite[Proposition 1]{PetitQ12}. They state that, by following this analysis the index calculus approach beats generic algorithms with run time $\mathcal{O}(2^{n/2})$ for any $n \geq N$ where $N$ is an integer approximately equal to $2000$. Now, based on \autoref{thm:semaev-first-fall} we assume $D_{reg} \approx m^2 - m + 1 = n^{2/3} - n^{1/3} + 1$ and sharpen \eqref{lbl:runtime} to
    	\begin{align}
    		\mathcal{O} \left( 2^{c \log(n) \left( n^{2/3} - n^{1/3} + 1 \right) } \right).
    	\end{align}
    Hence, the turning point to solve the ECDLP faster than a generic algorithm is an integer approximately equal to $1250$. Note that this is still far from cryptographically relevant sizes of $n$ up to $521$.
    
\bibliographystyle{amsplain}
\bibliography{firstfall}
    
\end{document}